\documentclass
{amsproc}

\usepackage{amsmath}
\usepackage{amsthm}
\usepackage{amssymb}
\usepackage{graphicx}
\usepackage{multirow}
\usepackage{tikz-cd}

\DeclareGraphicsExtensions{.png,.pdf,.eps}
\usepackage[all,2cell]{xy}
\usepackage{tikz}
\usepackage{pgf}
\usepackage{enumerate}
\usepackage{mathdots}
\usetikzlibrary{cd}
\usepackage{caption}
\usepackage{adjustbox}
\usetikzlibrary{matrix,arrows,backgrounds}

\usepackage{enumitem}
\setlist[itemize,enumerate,description]{leftmargin=*}

\usepackage{array}
\newcolumntype{H}{>{\setbox0=\hbox\bgroup}c<{\egroup}@{}}
\usepackage{hyperref}

\newtheorem{theorem}{Theorem}[section]
\newtheorem{lemma}[theorem]{Lemma}
\newtheorem{corollary}[theorem]{Corollary}
\newtheorem{proposition}[theorem]{Proposition}

\theoremstyle{definition}

\newtheorem{definition}[theorem]{Definition}

\newtheorem{setup}[theorem]{Set-up}

\newtheorem{remark}[theorem]{Remark}

\newtheorem{set}[theorem]{Setup}
\newtheorem{example}[theorem]{Example}

\title[Morelli-W\l odarczyk cobordism and examples of rooftop flips]{Morelli-W\l odarczyk cobordism and examples of rooftop flips}

\author[Barban]{Lorenzo Barban}
\address{Dipartimento di Matematica, Universit\`a degli Studi di Trento, via Sommarive 14, I-38123 Povo di Trento (TN), Italy}
\email{lorenzo.barban@unitn.it}

\author[Franceschini]{Alberto Franceschini}
\address{Dipartimento di Matematica, Universit\`a degli Studi di Trento, via Sommarive 14, I-38123 Povo di Trento (TN), Italy}
\email{alberto.bobech@gmail.com}


\thanks{}

\usepackage[textsize=tiny]{todonotes}

\usepackage{enumitem}

\newcommand\ignore[1]{}

\DeclareMathOperator{\HH}{H}

\def\C{{\mathbb C}}

\def\P{{\mathbb P}}
\def\Q{{\mathbb Q}}

\def\Z{{\mathbb Z}}

\def\cE{{\mathcal E}}

\def\cL{{\mathcal L}}

\def\cO{{\mathcal{O}}}

\def\cY{{\mathcal Y}}
\def\Q{{\mathbb{Q}}}

\def\operatorname#1{\mathop{\rm #1}\nolimits}

\def\Exc{\operatorname{Exc}}

\def\Pic{\operatorname{Pic}}

\def\Spec{\operatorname{Spec}}

\def\codim{\operatorname{codim}}

\def\deg{\operatorname{deg}}

\def\Nef{{\operatorname{Nef}}}

\def\NU{{\operatorname{N^1}}}

\newcommand{\pb}{\ar@{}[dr]|{\text{\pigpenfont J}}}

\makeatletter
\newcommand{\xleftrightarrow}[2][]{\ext@arrow 3359\leftrightarrowfill@{#1}{#2}}
\makeatother
\newcommand{\xdasharrow}[2][->]{
\tikz[baseline=-\the\dimexpr\fontdimen22\textfont2\relax]{
\node[anchor=south,font=\scriptsize, inner ysep=1.5pt,outer xsep=2.2pt](x){#2};
\draw[shorten <=3.4pt,shorten >=3.4pt,dashed,#1](x.south west)--(x.south east);
}}

\newcommand{\git}{\mathbin{/\mkern-6mu/}}

\begin{document}
\begin{abstract}
	We introduce the notion of rooftop flip, namely a small modification among normal projective varieties which is modeled by a smooth projective variety of Picard number 2 admitting two projective bundle structures. Examples include the Atiyah flop and the Mukai flop, modeled respectively by $\P^1\times \P^1$ and by $\P\left(T_{\P^2}\right)$. Using the Morelli-W\l odarczyk cobordism, we prove that any smooth projective variety of Picard number 1, endowed with a $\C^*$-action with only two fixed point components, induces a rooftop flip.
\end{abstract}
\maketitle
\tableofcontents



\section{Introduction}\label{sec:intro}
		
The relation between birational geometry and $\C^*$-actions has been deeply studied over the years (see for instance \cite{ReidFlip}, \cite{Thaddeus} and in recent years \cite{WORS1}).
One famous example for instance is the notion of \emph{Morelli-W\l odarczyk cobordism} (cfr. \cite[Definition 2]{Wlodarczyk}, see Definition \ref{def: cobordism}), introduced to study birational maps among normal projective varieties which are realized as geometric quotients by a $\C^*$-action. This algebraic version of cobordism has been used to prove the \emph{Weak factorization conjecture} (see \cite{Wlodarczyk}). 

On the other hand, the relation between birational maps and $\C^*$-actions becomes evident in the example of the \emph{Atiyah flop} (see \cite[\S 1.3]{ReidFlip}). 
Consider the $\C^*$-action on $\C^4$ given by 
\[
t \cdot \left(x_0,x_1,y_0,y_1\right)=\left( tx_0,tx_1,t^{-1}y_0,t^{-1}y_1\right),
\]
where $t\in\C^*$ and $(x_0,x_1,y_0,y_1)\in \C^4$. The GIT quotient $\C^4\to \C^4\git \C^*$ is the affine cone over the Segre embedding of $\P^1 \times \P^1$.
The variety $\C^4 \git \C^*$  has a cone singularity at the origin, which can be resolved by blowing-up the vertex of the cone.
The exceptional divisor $\P^1\times \P^1$ of the blow-up can be contracted in two different ways, producing two smooth varieties $X_1$ and $X_2$, which are isomorphic on the complement of a closed subset of codimension greater or equal than $2$. The resulting birational map $X_1 \dashrightarrow X_2$ is called the Atiyah flop and $\C^4$ is an example of cobordism associated to it. The following picture summarizes the example:
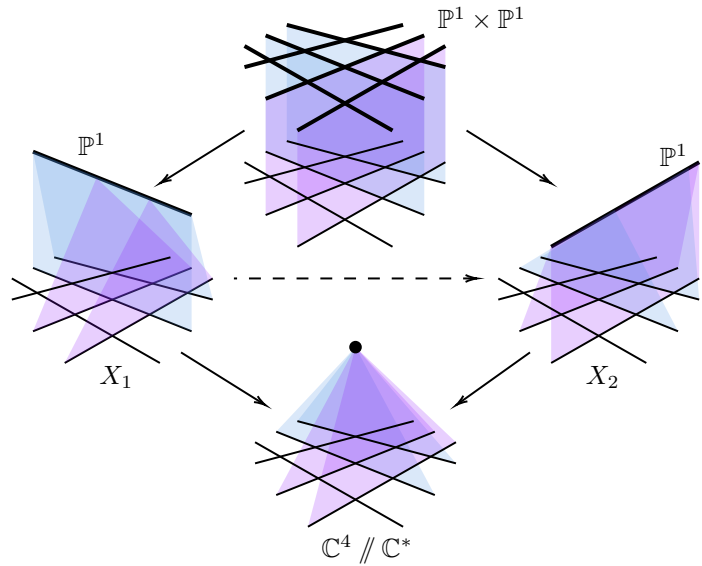
\begin{figure}[h!]
\centering
\tikzset{every picture/.style={line width=0.75pt}} 

\begin{tikzpicture}[x=0.4pt,y=0.4pt,yscale=-1,xscale=1]

\draw  [draw opacity=0][fill={rgb, 255:red, 144; green, 19; blue, 254 }  ,fill opacity=0.2 ] (660,150) -- (640,250) -- (490,310) -- (490,310) -- (520,230) -- cycle ;
\draw    (30,250) -- (180,310) ;
\draw    (180,250) -- (30,310) ;
\draw    (200,260) -- (60,340) ;
\draw [line width=1.5]    (30,140) -- (180,200) ;
\draw [line width=1.5]    (660,150) -- (520,230) ;
\draw  [draw opacity=0][fill={rgb, 255:red, 144; green, 19; blue, 254 }  ,fill opacity=0.2 ] (420,40) -- (420,150) -- (280,230) -- (280,230) -- (280,120) -- cycle ;
\draw  [draw opacity=0][fill={rgb, 255:red, 74; green, 144; blue, 226 }  ,fill opacity=0.2 ] (400,90) -- (400,200) -- (250,140) -- (250,140) -- (250,30) -- cycle ;
\draw  [draw opacity=0][fill={rgb, 255:red, 144; green, 19; blue, 254 }  ,fill opacity=0.2 ] (660,150) -- (660,260) -- (520,340) -- (520,340) -- (520,230) -- cycle ;
\draw  [draw opacity=0][fill={rgb, 255:red, 74; green, 144; blue, 226 }  ,fill opacity=0.2 ] (180,200) -- (200,280) -- (50,240) -- (50,240) -- (30,140) -- cycle ;
\draw  [draw opacity=0][fill={rgb, 255:red, 74; green, 144; blue, 226 }  ,fill opacity=0.2 ] (420,60) -- (420,170) -- (270,130) -- (270,130) -- (270,20) -- cycle ;
\draw  [draw opacity=0][fill={rgb, 255:red, 144; green, 19; blue, 254 }  ,fill opacity=0.2 ] (400,30) -- (400,140) -- (250,200) -- (250,200) -- (250,90) -- cycle ;
\draw  [draw opacity=0][fill={rgb, 255:red, 74; green, 144; blue, 226 }  ,fill opacity=0.2 ] (335,325) -- (410,465) -- (260,405) -- (260,405) -- (335,325) -- cycle ;
\draw  [draw opacity=0][fill={rgb, 255:red, 74; green, 144; blue, 226 }  ,fill opacity=0.2 ] (335,325) -- (430,435) -- (280,395) -- (280,395) -- (335,325) -- cycle ;
\draw  [draw opacity=0][fill={rgb, 255:red, 144; green, 19; blue, 254 }  ,fill opacity=0.2 ] (335,325) -- (430,415) -- (290,495) -- (290,495) -- (335,325) -- cycle ;
\draw  [draw opacity=0][fill={rgb, 255:red, 144; green, 19; blue, 254 }  ,fill opacity=0.2 ] (335,325) -- (410,405) -- (260,465) -- (260,465) -- (335,325) -- cycle ;
\draw    (170,330) -- (248.3,378.94) ;
\draw [shift={(250,380)}, rotate = 212.01] [color={rgb, 255:red, 0; green, 0; blue, 0 }  ][line width=0.75]    (10.93,-3.29) .. controls (6.95,-1.4) and (3.31,-0.3) .. (0,0) .. controls (3.31,0.3) and (6.95,1.4) .. (10.93,3.29)   ;
\draw    (431.63,378.84) -- (500,330) ;
\draw [shift={(430,380)}, rotate = 324.46] [color={rgb, 255:red, 0; green, 0; blue, 0 }  ][line width=0.75]    (10.93,-3.29) .. controls (6.95,-1.4) and (3.31,-0.3) .. (0,0) .. controls (3.31,0.3) and (6.95,1.4) .. (10.93,3.29)   ;
\draw  [dash pattern={on 4.5pt off 4.5pt}]  (448,260) -- (220,260) ;
\draw [shift={(450,260)}, rotate = 180] [color={rgb, 255:red, 0; green, 0; blue, 0 }  ][line width=0.75]    (10.93,-3.29) .. controls (6.95,-1.4) and (3.31,-0.3) .. (0,0) .. controls (3.31,0.3) and (6.95,1.4) .. (10.93,3.29)   ;
\draw  [draw opacity=0][fill={rgb, 255:red, 74; green, 144; blue, 226 }  ,fill opacity=0.2 ] (650,155) -- (660,280) -- (660,280) -- (510,240) -- (510,240) -- cycle ;
\draw  [draw opacity=0][fill={rgb, 255:red, 74; green, 144; blue, 226 }  ,fill opacity=0.2 ] (580,195) -- (640,310) -- (640,310) -- (490,250) -- (490,250) -- cycle ;
\draw    (151.7,168.94) -- (230,120) ;
\draw [shift={(150,170)}, rotate = 327.99] [color={rgb, 255:red, 0; green, 0; blue, 0 }  ][line width=0.75]    (10.93,-3.29) .. controls (6.95,-1.4) and (3.31,-0.3) .. (0,0) .. controls (3.31,0.3) and (6.95,1.4) .. (10.93,3.29)   ;
\draw    (440,120) -- (518.3,168.94) ;
\draw [shift={(520,170)}, rotate = 212.01] [color={rgb, 255:red, 0; green, 0; blue, 0 }  ][line width=0.75]    (10.93,-3.29) .. controls (6.95,-1.4) and (3.31,-0.3) .. (0,0) .. controls (3.31,0.3) and (6.95,1.4) .. (10.93,3.29)   ;
\draw  [fill={rgb, 255:red, 0; green, 0; blue, 0 }  ,fill opacity=1 ] (330,325) .. controls (330,322.24) and (332.24,320) .. (335,320) .. controls (337.76,320) and (340,322.24) .. (340,325) .. controls (340,327.76) and (337.76,330) .. (335,330) .. controls (332.24,330) and (330,327.76) .. (330,325) -- cycle ;
\draw    (240,415) -- (380,495) ;
\draw    (430,415) -- (290,495) ;
\draw    (410,405) -- (260,465) ;
\draw    (390,395) -- (240,435) ;
\draw    (260,405) -- (410,465) ;
\draw    (280,395) -- (430,435) ;
\draw    (50,240) -- (200,280) ;
\draw  [draw opacity=0][fill={rgb, 255:red, 74; green, 144; blue, 226 }  ,fill opacity=0.2 ] (180,200) -- (180,310) -- (30,250) -- (30,250) -- (30,140) -- cycle ;
\draw    (10,260) -- (150,340) ;
\draw  [draw opacity=0][fill={rgb, 255:red, 144; green, 19; blue, 254 }  ,fill opacity=0.2 ] (200,260) -- (200,260) -- (60,340) -- (60,340) -- (140,185) -- cycle ;
\draw  [draw opacity=0][fill={rgb, 255:red, 144; green, 19; blue, 254 }  ,fill opacity=0.2 ] (180,250) -- (180,250) -- (30,310) -- (30,310) -- (90,165) -- cycle ;
\draw    (160,240) -- (10,280) ;
\draw    (620,240) -- (470,280) ;
\draw    (640,250) -- (490,310) ;
\draw    (660,260) -- (520,340) ;
\draw    (470,260) -- (610,340) ;
\draw    (490,250) -- (640,310) ;
\draw    (510,240) -- (660,280) ;
\draw [line width=1.5]    (270,20) -- (420,60) ;
\draw [line width=1.5]    (250,30) -- (400,90) ;
\draw [line width=1.5]    (230,40) -- (370,120) ;
\draw [line width=1.5]    (380,20) -- (230,60) ;
\draw [line width=1.5]    (400,30) -- (250,90) ;
\draw [line width=1.5]    (420,40) -- (280,120) ;
\draw    (230,150) -- (370,230) ;
\draw    (250,140) -- (400,200) ;
\draw    (270,130) -- (420,170) ;
\draw    (420,150) -- (280,230) ;
\draw    (400,140) -- (250,200) ;
\draw    (380,130) -- (230,170) ;

\draw (410,0) node [anchor=north west][inner sep=0.75pt]    {$\P^1 \times \P^1$};
\draw (620,130) node [anchor=north west][inner sep=0.75pt]    {$\mathbb{P}^{1}$};
\draw (70,120) node [anchor=north west][inner sep=0.75pt]    {$\mathbb{P}^{1}$};
\draw (300,500) node [anchor=north west][inner sep=0.75pt]    {$\C^4 \git \C^*$};
\draw (90,340) node [anchor=north west][inner sep=0.75pt]    {$X_1$};
\draw (550,340) node [anchor=north west][inner sep=0.75pt]    {$X_2$};

\end{tikzpicture}
\end{figure}

\noindent
One may generalize this example by considering a similar $\C^*$-action on $\C^{n+1}$ (see for instance \cite[Example 1]{Wlodarczyk}), or allowing also weights different from $\pm 1$ (cf. \cite[\S 3]{BR}): the resulting birational map is called the \emph{Atiyah flip}.
In the above construction of the Atiyah flop, the varieties $X_1$ and $X_2$ are defined using the two projective bundle structures on $\P^1 \times \P^1$: 
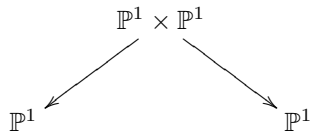
\begin{equation*}\label{eq:rooftopP1xP1}
\xymatrix{&\P^1 \times \P^1 \ar[ld]_{} \ar[rd]^{} & \\ \P^1 && \P^1}
\end{equation*}
Smooth projective varieties with two different projective bundle structures have been already studied in the literature (see for instance \cite{ORS}): in particular, in \cite[Lemma 4.4]{WORS1} the authors have constructed a correspondence between them and smooth projective varieties $X$ of Picard number $1$ admitting a $\C^*$-action having only two fixed point components. 

Motivated by the example of Atiyah flop and its connection with a variety admitting two projective bundle structures, we introduce the notion of \emph{rooftop flip}.
Given a smooth projective variety $\Lambda$ of Picard number $2$ admitting two projective bundle structures 
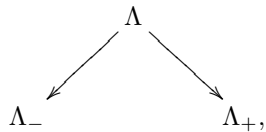
\begin{equation*}\label{eq:rooftopP1xP1}
	\xymatrix{&\Lambda \ar[ld]_{} \ar[rd]^{} & \\ \Lambda_- && \Lambda_+,}
\end{equation*}
a \emph{rooftop flip modeled by $\Lambda$} (see Definition \ref{def:Dflips}) is a small modification $\psi: W_-\dashrightarrow W_+$ among normal projective varieties which can be resolved by a common divisorial extraction 

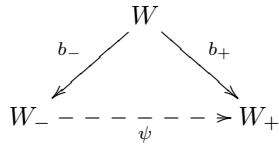
\[
\xymatrix{ &W \ar[ld]_{b_-} \ar[rd]^{b_+} & \\ W_- \ar@{-->}[rr]_\psi && W_+}
\]
such that the restriction of $b_{\pm}: W\to W_{\pm}$ to the exceptional loci are modeled by the projective bundle maps $p_{\pm}: \Lambda\to \Lambda_{\pm}$.

In this setting, the Atiyah flop is a rooftop flip modeled by $\P^1\times \P^1$.  As we will see, the class of rooftop flips includes classical birational transformations: the Atiyah flip is a rooftop flip modeled by $\P^m\times \P^l$, and the Mukai flop (see \cite{HZ04}, \cite{WW}) is a rooftop flip modeled by $\P\left(T_{\P^2}\right)$. Moreover, rooftop flips appear as local models of the birational transformations of the geometric quotients of a smooth polarized pair $(X,L)$ under a $\C^*$-actions (see \cite[Theorem 1.1 (1)]{WORS3})
In the paper we discuss the definition of  rooftop flip and its connection with smooth projective varieties having a $\C^*$-action with only two fixed point components.
The main result of the paper is the following:
		
\begin{theorem}\label{theorem:naive}
	Given a smooth projective variety $\Lambda$ of Picard number $2$ with two projective bundle structures, there exist two quasi-projective varieties and a rooftop flip modeled by $\Lambda$ among them.
\end{theorem}

\subsection*{Outline}

In Section \ref{sec:preliminaries} we first recall some notions regarding $\C^*$-actions on smooth polarized pairs. We then focus on the case of $\C^*$-action whose associated fixed point locus consists only of $2$ connected components; such varieties are called \emph{drums} (see Definition \ref{definition:drum}) and they are constructed upon the choice of a triple $(Y,\cL_-,\cL_+)$, where $Y$ is a smooth projective variety of Picard number $2$ admitting two projective bundle structure, and $\cL_{\pm}$ are semiample line bundles on $Y$. We conclude by recalling the relation between smooth drums and the classification of horospherical varieties done by \cite{Pas} (see Remark \ref{remark:pasquier}).

In Section \ref{sec:atiyah}, after recalling the Morelli-W\l odarczyk cobordism (see Definition \ref{def: cobordism}), we introduce and discuss the notion of rooftop flip (see Definition \ref{def:Dflips}). We then revisit the example of Atiyah flip and prove that it is indeed a rooftop flip (see \S \ref{ssec:atiyah}).

In Section \ref{sec:mainresult} we prove Theorem \ref{theorem:naive}. The idea of the proof is to view our case as a restriction of the Morelli-W\l odarczyk cobordism associated to the Atiyah flip. We conclude by deducing that the drum structure on the smooth quadric hypersurface induces a rooftop flip modeled by $\P\left(T_{\P^n}\right)$, see Corollary \ref{corollary:rooftopflipquadric}. 

\subsection*{Acknowledgements} We would like to thank Luis E. Sol\'a Conde for having suggested this problem, for all the stimulating conversations and the important suggestions. 

\section{Preliminaries}\label{sec:preliminaries}

\subsection{Notation}

We work over $\C$. In this paper $X$ will be a smooth projective variety. We denote by $\rho_X:= \dim \NU(X)$ the Picard number of $X$, where $\NU(X)$ is the finite-dimensional real vector space of Cartier divisors modulo numerical equivalence. 
Given $V$ a finite-dimensional vector space, by $\P\left(V\right)$ we mean the space of one-dimensional quotients of $V$.
The standard $\C^*$-action by homoteties is denoted by $\C^*_h$, with coordinate $h$. For sake of notation, by $t$ we will denote the coordinate of a $\C^*$-action which does not act by homoteties.  
By a \emph{contraction} we mean a proper surjective morphism with connected fibers.
A birational contraction $f$ is \emph{small} if $\codim \Exc(f)\geq 2$. 
By a \emph{flip} we mean a $D$-flip as in \cite[Definition 5.1.4]{hacon}.

\subsection{$\C^*$-actions on smooth projective varieties}\label{ssec:actions}

In this section, we collect some preliminaries about $\C^*$-actions on smooth projective varieties. We refer to \cite[Section 2]{WORS1} for a detailed discussion.

Let $X$ be a smooth projective variety endowed with a $\C^*$-action. By $X^{\C^*}$ we denote the fixed point locus of the action, and by $\cY$ we denote its set of connected components, namely:
$$X^{\C^*}=\bigsqcup_{Y \in \cY} Y.$$
By \cite[Main Theorem]{IVERSEN}, the subvarieties $Y\in \cY$ are smooth and irreducible.

For any $Y\in \cY$, we define the \emph{Bia{\l}ynicki-Birula cells} as
\[
X^{\pm}(Y):=\left\{x\in X:\lim_{t\to 0}t^{\pm 1}\cdot x\in Y\right\}.
\]
The \emph{sink} (resp. \emph{source}) of the $\C^*$-action is the unique component $Y_-\in \cY$ (resp. $Y_+\in\cY$) such that $X^-(Y_-)$ (resp. $X^+(Y_+)$) are dense subsets of $X$.
As noticed in \cite[Remark 2.5]{WORS1}, the uniqueness of the sink and the source follows from the well-known Bia\l ynicki-Birula Theorem (see \cite{BB} for the original exposition).
 
\begin{definition}
	A $\C^*$-action is \emph{equalized} at $Y\in \cY$ if for every point $p\in \left(X^-(Y)\cup X^+(Y)\right)\setminus Y$ the isotropy group of the $\C^*$-action at $p$ is trivial. Moreover, a $\C^*$-action is equalized if it is equalized at every fixed point component $Y\in \cY$. 
\end{definition}

Let $L$ be an ample line bundle on $X$.
Consider the induced $\C^*$-action on $L$; by \cite[Proposition 2.4 and the subsequent Remark]{KKLV}, a linearization of $L$ exists. 
For any $Y\in \cY$, the induced action of $\C^*$ on the fibers of $L|_{Y}$ is by multiplication with a character, which we denote by $\mu_L(Y) \in \Z$. 
Since $L$ is ample, one may show that
\[
\mu_L(Y_-)=\min_{Y\in\cY}\mu_L(Y), \qquad \mu_L(Y_+)=\max_{Y\in\cY}\mu_L(Y).
\]
By a \emph{$\C^*$-action on the smooth polarized pair $(X,L)$} we mean a non-trivial $\C^*$-action on a smooth projective variety $X$ and a linearization of the ample line bundle $L$.

\begin{remark}
	A $\C^*$-action on $(X,L)$ induces a $\C^*$-action on $(X,mL)$ for every $m\geq 0$. 
	In particular, for $m\gg 0$, the line bundle $mL$ is very ample, and we have an embedding $\phi_{|mL|}: X\hookrightarrow \P\left(\HH^0(X,mL)\right)$. 
	Following \cite[Section 2]{BWW}, a linearization of $mL$ provides a $\C^*$-action on $\HH^0(X,mL)$, and so on $\P \left(\HH^0(X,mL)\right)$.
	Moreover, the $\C^*$-action on $X$ can be seen as the restriction of the $\C^*$-action on $\P\left(\HH^0(X,mL)\right)$, that is the embedding is $\C^*$-equivariant.
\end{remark}

\begin{definition}
	Let $X$ be a smooth projective variety and let $L\in \Pic(X)$ be ample. 
The \emph{bandwidth} of the $\C^*$-action on $(X,L)$ is defined as 
\[
|\mu_L|:=\mu_L(Y_+)-\mu_L(Y_-).
\]
\end{definition}

\subsection{Smooth drums}\label{ssec:drums}

In this section we recall the characterization of smooth drums, that is smooth polarized pairs admitting a $\C^*$-action of bandwidth $1$. We refer to \cite[Section 4]{WORS1}. 

\begin{lemma}\label{lemma:gg}
	Let $Y$ be a smooth projective variety such that $\rho_Y=2$, admitting two elementary contractions $p_{\pm}:Y\to Y_{\pm}$, with $Y_{\pm}$ smooth. Let $L_{\pm}$ be very ample line bundle on $Y_{\pm}$, and denote $\cL_{\pm}:=p_{\pm}^*L_{\pm}$. Consider the projective bundle $\pi:\P(\cL_-\oplus \cL_+)\to Y$. Then $\cO_{\P(\cL_-\oplus \cL_+)}(1)$ is globally generated, and there exists a contraction, birational onto the image,
	\[
\phi=\phi_{\cO_{\P(\cL_-\oplus \cL_+)}(1)}:\P\left(\cL_-\oplus \cL_+\right)\longrightarrow \P\left(\HH^0\left(\P(\cL_-\oplus \cL_+), \cO_{\P(\cL_-\oplus \cL_+)}(1)\right)\right).
\]
\end{lemma}  

\begin{proof}
	The global generation of $\cO_{\P(\cL_-\oplus \cL_+)}(1)$ is immediate. Let us just notice that, by the projection formula, we have an isomorphism
	\begin{equation}\label{equation:drumembedding}
		\HH^0\left(\P(\cL_-\oplus \cL_+),\cO_{\P(\cL_-\oplus \cL_+)}(1)\right)= \HH^0(Y_-,L_-)\oplus\HH^0(Y_+,L_+).
	\end{equation}
	Let us prove that the morphism 
	\[
	\phi: \P(\cL_-\oplus \cL_+) \to \P(\HH^0(Y_-,L_-)\oplus \HH^0(Y_+,L_+)),
	\]
	associated to evaluation of sections is a contraction, birational onto the image.
	 Consider the sections $\sigma_{\pm}: Y\to \P\left(\cL_-\oplus \cL_+\right)$ associated to the quotients $\cL_-\oplus \cL_+\to \cL_{\pm}$. The compositions $\phi\circ \sigma_{\pm}$ coincide with the bundle maps $p_{\pm}$, in particular they have connected fibers. 
	 On the other hand the restriction of $\phi$ to $\P\left(\cL_-\oplus \cL_+\right)\setminus \left(\sigma_-(Y_-)\cup \sigma_+(Y)\right)$ is an isomorphism onto the image.	
\end{proof}

\begin{definition}\label{definition:drum}
	The image $X:=\phi \left(\P(\cL_-\oplus \cL_+)\right)$ is called the \emph{drum constructed upon the triple $(Y,\cL_-,\cL_+)$}.
\end{definition}

\begin{remark}\label{remark:drumembedding}
	Using Equation \ref{equation:drumembedding}, we obtain that there is an embedding 
\[
X \subset \P\left(\HH^0(Y_-,L_-)\oplus \HH^0(Y_-,L_+)\right).
\]
\end{remark}

\begin{remark}\label{remark:drumdimension}
	By construction we have that $\dim X=\dim Y +1.$
\end{remark}

\begin{theorem}\label{theorem:smoothdrum}\cite[Lemma 4.4]{WORS1}
	A drum $X$ is smooth if and only if the following conditions are satisfied:
	\begin{itemize}
		\item $\Nef(Y)=\langle \cL_-,\cL_+ \rangle$;
		\item $p_\pm: Y\to  Y_\pm$ has a projective bundle structure;
		\item $\deg(\cL_{\mp}|_{F_{\pm}})=1$,  where $F_\pm$ denotes a fiber of $p_\pm$.
	\end{itemize}
\end{theorem}

\begin{remark}
Since $\rho_{Y_\pm}=1$, the third condition of the above theorem implies that $L_\pm$ must be the respective generators of the Picard groups of $Y_\pm$.
\end{remark}

Since we will work only in the context of smooth drums, we fix the notation for the rest of the section in the following:

\begin{setup}\label{setup:smoothdrum}
	Let $Y$ be a smooth projective variety such that $\rho_Y=2$, admitting two projective bundle structure $p_{\pm}:Y\to Y_{\pm}$, with $Y_{\pm}$ smooth. Let $L_{\pm}$ be very ample line bundles on $Y_{\pm}$, and denote $\cL_{\pm}:=p_{\pm}^*L_{\pm}$. Suppose that $\deg(\cL_{\mp}|_{F_{\pm}})=1$,  where $F_\pm$ denotes a fiber of $p_\pm$. Consider the projective bundle $\pi:\P(\cL_-\oplus \cL_+)\to Y$ and the morphism $\phi$ given by $\cO_{\P(\cL_-\oplus \cL_+)}(1)$. Then let $X$ be the smooth drum constructed upon $(Y,\cL_-,\cL_+)$.
\end{setup}

\begin{remark}
	Notice that $X$ comes with a natural ample line bundle $L$, which is the restriction of the hyperplane class in $\P(\cL_-\oplus\cL_+)$, such that $\phi^*L=\cO_{\P(\cL_-\oplus\cL_+)}(1)$.
\end{remark}

For sake of clarity, Set-up \ref{setup:smoothdrum} can be summarized by means of the following diagram:
\[
\xymatrix{ & \P\left(\cL_-\oplus \cL_+\right) \ar[r]^-\phi \ar[d]_\pi & X \\  L_- \ar[d]& Y \ar[ld]_{p_-}  \ar[rd]^{p_+} &L_+ \ar[d] \\ Y_- && Y_+ }
\]

\begin{lemma}
Let $\pi: \P\left(\cL_-\oplus \cL_+\right)\to Y$ be a projective bundle as in Set-up \ref{setup:smoothdrum}. Then $\P\left(\cL_-\oplus \cL_+\right)$ admits a $\C^*$-action of bandwidth $1$, whose sink and source are both isomorphic to $Y$.
\end{lemma}

\begin{proof}
Consider the \emph{standard $\C^*$-action} on $\left(\P^1, \cO_{\P^1}(1)\right)$: 
\[
\C^* \times \P^1 \to \P^1, \quad \left(t, \left[x_0:x_1\right]\right) \longmapsto \left[tx_0:x_1\right] .
\]
This action has bandwidth $1$. Let $\sigma_\pm:Y \to \P\left(\cL_-\oplus\cL_+\right)$ be the sections corresponding to the quotients $\cL_-\oplus\cL_+ \to \cL_\pm$.
Then the $\P^1$-bundle structure on $\P\left(\cL_-\oplus\cL_+\right)$ allows us to define a $\C^*$-action given fiberwise by the standard $\C^*$-action on $\P^1$ and such that the sink is $\sigma_-(Y)$ and the source is $\sigma_+(Y)$.
\end{proof}

\begin{remark}\label{remark:sections}\cite[Remark 4.2]{WORS1}
Let $X$ be a smooth drum constructed upon $\left(Y,\cL_-,\cL_+\right)$ as in Set-up \ref{setup:smoothdrum}. We recall that, by Remark \ref{remark:drumembedding}, $\phi$ provides an embedding $X \subset \P\left(\HH^0(Y_-,L_-)\oplus \HH^0(Y_-,L_+)\right)$.
As we note in the previous proof, $\sigma_\pm(Y)$ contain the limit points for $t^{\pm 1} \to 0$. Thus we have that:
\begin{itemize}
\item $(\phi \circ \sigma_\pm)(Y) \simeq Y_\pm$, embedded in $\P\left(\HH^0(Y_\pm,L_\pm)\right)$ via $\cO_{\P(\cL_-\oplus \cL_+)}(1)$;
\item For all $y \in Y$, $\pi^{-1}(y)$ is mapped by $\phi$ into a $\C^*$-orbit.
\end{itemize}
Then $X$ admits a $\C^*$-action of bandwidth $1$, whose sink and source are $Y_-$ and $Y_+$, respectively.
\end{remark}

\begin{example}\label{example:drumPn}
Consider $Y_-=\P^m, Y_+=\P^l$, let $Y=\P^m \times \P^l$, and let $p_\pm$ be the projections onto the two factors of $\P^m \times \P^l$. Denote by $L_-=\cO_{\P^m}(1)$ and $L_+=\cO_{\P^l}(1)$ the very ample line bundles on $\P^m$ and $\P^l$, respectively.
Then $\cL_-=\cO_{\P^m \times \P^l}(1,0)$, $\cL_+=\cO_{\P^m \times \P^l}(0,1)$, and we have a diagram of the form
\[
\xymatrix{ &\P\left(\cO_{\P^m\times\P^l}(1,0) \oplus \cO_{\P^m\times\P^l}(0,1)\right) \ar[r]^-\phi \ar[d]_\pi & X \\  \cO_{\P^m}(1) \ar[d]& \P^m\times\P^l \ar[ld]_{p_-}  \ar[rd]^{p_+} &\cO_{\P^l}(1) \ar[d] \\ \P^m && \P^l }
\]
As shown in Remark \ref{remark:drumembedding}, $X \subset \P\left( \HH^0(\P^m,\cO_{\P^m}(1))\oplus \HH^0(\P^l,\cO_{\P^l}(1))\right) \simeq \P^{m+l+1}$.
By Remark \ref{remark:drumdimension}, $\dim X= \dim \P^m\times \P^l +1$, hence $\phi$ is surjective and $X \simeq \P^{m+l+1}$. 
In particular, $\P^{m+l+1}$ is the drum constructed upon 
\[
\left(\P^m\times \P^l, \cO_{\P^m\times\P^l}(1,0),\cO_{\P^m\times \P^l}(0,1)\right).
\]
On the other hand, consider the $\C^*$-action on $\P^{m+l+1} $ given by
	\begin{equation*}\label{eq:actionP}
	t \cdot \left[x_0:\ldots:x_{m+l+1}\right]=\left[tx_0:\ldots:t x_{l}:x_{l+1}:\ldots:x_{m+l+1}\right].
	\end{equation*}
	The sink and the source are, respectively, 
\[
\P^m=\left\{x_{l+1}=\ldots=x_{m+l+1}=0\right\},\qquad \P^l=\left\{x_0=\ldots=x_l=0\right\}.
\]
Then the $\C^*$-action of bandwidth $1$ on $X$ induced by the drum structure is precisely the action defined above.
\end{example}

\begin{theorem}\label{thm:bw1}\cite[Theorem 4.8]{WORS1}
Let $X$ be a smooth projective variety with $\rho_X=1$ different from the projective space and let $L$ be an ample line bundle on $X$. Then $(X,L)$ admits a $\C^*$-action of bandwidth $1$ if and only if $X$ is a smooth drum.
\end{theorem}

\begin{example}\label{example:quadric}
	Consider the $\C^*$-action on $\P^{2n+1}$ defined as follows:
	\[
	t \cdot \left[x_0:\ldots:x_{2n+1}\right]=\left[tx_0:\ldots:tx_n: x_{n+1}:\ldots:x_{2n+1}\right].
	\]
	Let $Q^{2n} \subset \P^{2n+1}$ be the smooth quadric hypersurface defined by the equation $x_0x_{n+1}+x_1x_{n+2}+\ldots+x_nx_{2n+1}=0$. By construction $Q^{2n}$ is $\C^*$-invariant and its fixed point locus consists only of the sink and the source, which are respectively:
	\[
	Y_-=\left\{x_{n+1}=\ldots=x_{2n+1}=0\right\} \simeq \P^n, \qquad Y_+=\left\{x_0=\ldots=x_n=0\right\} \simeq (\P^n)^\vee,
	\]
	where the duality between $Y_-$ and $Y_+$ is provided by the non-degenerate quadratic form defining $Q^{2n}$.
	Since the $\C^*$-action on $Q^{2n}$ has bandwidth $1$, using Theorem \ref{theorem:smoothdrum} we argue that $Q^{2n}$ is a smooth drum. By \cite[Proposition 1.9]{Pas}, the smooth projective variety of Picard number $2$ inducing this drum is 
	\begin{equation*}
		\xymatrix{&\P\left(T_{\P^n}\right) \ar[ld]_{p_-} \ar[rd]^{p_+} & \\
			\P^n && (\P^n)^\vee}
	\end{equation*}
	where the two projective bundle structures are induced by the well known description $\P\left(T_{\P^n}\right)=\left\{(p,H)\in \P^n\times (\P^n)^\vee\mid p\in H\right\}$. 
	Let us briefly explain this fact; let $X$ be the smooth drum associated to $\P\left(T_{\P^n}\right)$. By Remark \ref{remark:drumembedding}, $X \subset \P^{2n+1}$ and, by Remark \ref{remark:sections}, the sink and the source of the $\C^*$-action on $X$ are respectively $\P^n$, $(\P^n)^\vee$. Write $\P^n=\P(V)$, $(\P^n)^\vee=\P(V^\vee)$, with $V$ a complex vector space of dimension $n+1$. By construction the drum $X$ of $\P(T_{\P^n})$ is an hypersurface in $\P(V\oplus V^\vee)$. A point in $\P(V\oplus V^\vee)$ belongs to $X$ if it is the class of a vector $p+ h\in V\oplus V^\vee$ such that $h(p)=0$. Algebraically this says that $X$ is given by a nondegenerate quadratic equation, hence we conclude.
	
\end{example}

\begin{remark}\label{remark:pasquier}
To our understanding, the only known examples of smooth drums are constructed upon a smooth projective variety $Y$ admitting two projective bundle structures such that $Y_\pm$ are rational homogeneous varieties, i.e. quotient of semisimple linear algebraic groups by parabolic subgroups.
Except for a non-homogeneous sporadic example appearing in \cite{Ott2} (see also \cite[\S 2]{Kan}), in all the known examples $Y$ is also a rational homogeneous variety.
In \cite[Remark 3.3]{ORS} there is a complete classification of rational homogeneous varieties with Picard number $2$ admitting two projective bundle structures.
The corresponding smooth drums $X$ constructed from a rational homogeneous variety $Y$ are well-known in literature: they are one of the horospherical varieties classified by Pasquier, see  \cite[Theorem 0.1]{Pas}.
\end{remark}


\section{Morelli-W\l odarczyk cobordism and rooftop flips}\label{sec:atiyah}

In this section we first recall the definition of \emph{Morelli-W\l odarczyk cobordism} and we introduce the notion of \emph{rooftop flip}. We will then concentrate on studying the Atiyah flip; the reason is two-fold: on one hand this birational transformation, as we will see, can be constructed in terms of the Morelli-W\l odarczyk cobordism; on the other hand, it is the motivating example of the notion of rooftop flip.

\begin{definition}\label{def: cobordism}
	Let $X_1,X_2$ be birationally equivalent normal varieties. The \emph{Morelli--W\l odarczyk cobordism} between $X_1$ and $X_2$ is a normal variety $B$, endowed with a $\C^*$-action such that 
	\begin{equation*}
		\begin{split}
			B_+&=\{p\in B\mid \lim_{t\to 0} tp \text{ does not exists}\},\\
			B_-&=\{p\in B\mid \lim_{t\to \infty} tp \text{ does not exists}\}
		\end{split}
	\end{equation*}
	are non-empty open subsets of $B$, such that there exist geometric quotients $B_-/\C^*$ and $B_+/\C^*$ satisfying
	$$B_-/\C^*\simeq X_1 \dashrightarrow X_2\simeq B_+/\C^*,$$
	where the birational equivalence is realized by the open subset $(B_-\cap B_+)/\C^*$ contained in $B_{\pm}/\C^*$.
\end{definition}

We stress that the notation in the above Definition is slightly different from the original one (cf. \cite[Definition 2]{Wlodarczyk}), in particular the role on $B_-$ and $B_+$ are switched. The reason behind this apparent misleading decision is that in this setting it will be less confusing in the next section to keep track of the $\pm$-signs.

The following definition is the core of the section.
Note that a similar definition appears in the contest of \emph{$K$-equivalence}, see for instance \cite{Kan2}.

\begin{definition}\label{def:Dflips}
	Consider a normal projective variety $\Lambda$ with $\rho_{\Lambda}=2$ admitting two projective bundle structures:
	\[
	\xymatrix{ & \Lambda \ar[ld]_{p_-} \ar[rd]^{p_+} & \\ \Lambda_- && \Lambda_+}
	\]
	A small modification $\psi: W_- \dashrightarrow W_+$ between normal projective varieties is called a \emph{rooftop flip modeled by $\Lambda$} if the following holds: 
	\begin{enumerate}
		\item There are small contractions $s_\pm: W_{\pm}\to W_0$, with $W_0$ a normal projective variety,
			\[
			\xymatrix{W_- \ar@{-->}[rr]^\psi \ar[rd]_{s_-} && W_+, \ar[ld]^{s_+} \\ & W_0 & }
			\]  such that, denoting by $Z_\pm \subset W_\pm$ their exceptional loci, the restrictions $s_\pm|_{Z_\pm}: Z_\pm \to Z_0 \subset W_0$ are smooth and the fibers are $\Lambda_\pm$-bundles.
		\item There is a resolution
			
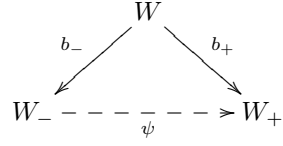
\[
			\xymatrix{ &W \ar[ld]_{b_-} \ar[rd]^{b_+} & \\ W_- \ar@{-->}[rr]_\psi && W_+}
			\]
			such that $Z:=b_\pm^{-1}(Z_\pm) \subset W$ is a divisor, and $b_\pm|_Z: Z \to Z_\pm$ defines projective bundle structures on $Z$.
	\item For any $z_0 \in Z_0$ we have that $b^{-1}_\pm|_{s^{-1}_\pm(z_0)}=p_\pm^{-1}$:
	
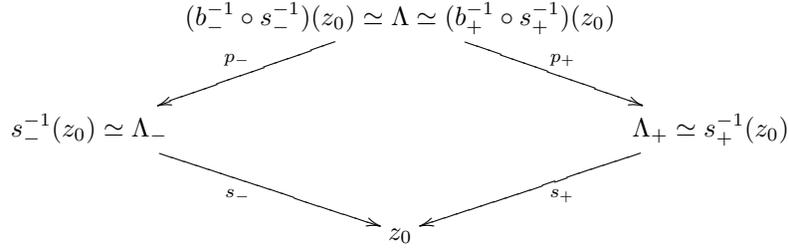
\[
	\xymatrix@C0.5pt{ &(b_-^{-1}\circ s_-^{-1})(z_0)\simeq \Lambda \simeq (b_+^{-1}\circ s_+^{-1})(z_0) \ar[ld]_{p_-} \ar[rd]^{p_+} & \\ s_-^{-1}(z_0) \simeq \Lambda_- \ar[rd]_{s_-} 
		&& \Lambda_+ \simeq s_+^{-1}(z_0) \ar[ld]^{s_+} \\ & z_0& }
	\]
	\end{enumerate}	
\end{definition}

\begin{remark}\label{remark:why}
	As we will see, the definition of rooftop flip includes some classical birational transformations: the Atiyah flip is a rooftop flip modeled by $\P^m\times \P^l$ (see \S\ref{ssec:atiyah}), and the classic Mukai flop (see \cite{HZ04}, \cite{WW}) is a rooftop flip modeled by $\P\left(T_{\P^2}\right)$ (see Corollary \ref{corollary:rooftopflipquadric}).
\end{remark}

\subsection{Atiyah rooftop flips}\label{ssec:atiyah}

In this section we show that the Atiyah flip is in particular a rooftop flip modeled by $\P^{m}\times \P^{l}$.

\begin{setup}\label{setup:atiyah}
	Let $V_-$ and $V_+$ denote the complex vector spaces of dimension respectively $m+1$ and $l+1$, with $l,m\ge 1$, and set $V:=V_-\oplus V_+$.
	Consider the $\C^*$-action on $V$ having weight $-1$ on $V_-$ and weight $1$ on $V_+$. 
	There is an induced $\C^*$-action on $V^\vee$ given by $t\cdot v= \left(t v_-,t^{-1}v_+\right)$, where $v=\left(v_-,v_+\right)\in V^\vee$.
	We will frequently abuse notation by writing $V_{-}$ (resp. $V_+$) for $V_-\times \{0\}$ (resp. $\{0\}\times V_+$).
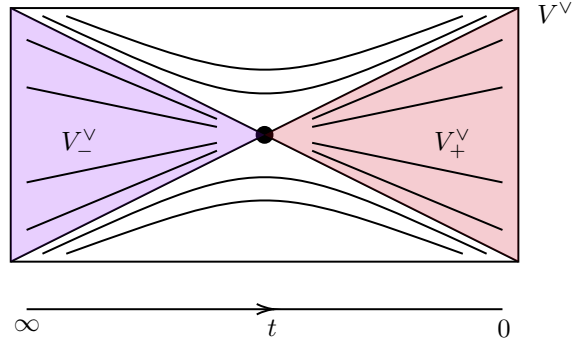
\begin{figure}[h!]
	\centering

\tikzset{every picture/.style={line width=0.75pt}} 

\begin{tikzpicture}[x=0.6pt,y=0.6pt,yscale=-1,xscale=1]

\draw   (20,20) -- (340,20) -- (340,180) -- (20,180) -- cycle ;
\draw    (20,20) -- (340,180) ;
\draw    (20,180) -- (340,20) ;
\draw    (30,210) -- (330,210) ;
\draw [shift={(186,210)}, rotate = 180] [color={rgb, 255:red, 0; green, 0; blue, 0 }  ][line width=0.75]    (10.93,-3.29) .. controls (6.95,-1.4) and (3.31,-0.3) .. (0,0) .. controls (3.31,0.3) and (6.95,1.4) .. (10.93,3.29)   ;
\draw  [fill={rgb, 255:red, 0; green, 0; blue, 0 }  ,fill opacity=1 ] (175,100) .. controls (175,97.24) and (177.24,95) .. (180,95) .. controls (182.76,95) and (185,97.24) .. (185,100) .. controls (185,102.76) and (182.76,105) .. (180,105) .. controls (177.24,105) and (175,102.76) .. (175,100) -- cycle ;
\draw  [draw opacity=0][fill={rgb, 255:red, 144; green, 19; blue, 254 }  ,fill opacity=0.2 ] (180,100) -- (20,180) -- (20,20) -- cycle ;
\draw  [draw opacity=0][fill={rgb, 255:red, 208; green, 2; blue, 27 }  ,fill opacity=0.2 ] (180,100) -- (340,20) -- (340,180) -- cycle ;
\draw    (30,40) -- (150,90) ;
\draw    (30,70) -- (150,95) ;
\draw    (30,160) -- (150,110) ;
\draw    (30,130) -- (150,105) ;
\draw    (210,95) -- (330,70) ;
\draw    (210,90) -- (330,40) ;
\draw    (210,110) -- (330,160) ;
\draw    (210,105) -- (330,130) ;
\draw    (40,25) .. controls (180.33,90.17) and (180.33,90.17) .. (320,25) ;
\draw    (55,25) .. controls (175,70.17) and (185.33,70.17) .. (305,25) ;
\draw    (40,175) .. controls (180,110.75) and (180,110.5) .. (320,175) ;
\draw    (55,175) .. controls (180,130.17) and (180.33,130.17) .. (305,175) ;

\draw (180,215) node [anchor=north west][inner sep=0.75pt]    {$t$};
\draw (20,215) node [anchor=north west][inner sep=0.75pt]    {$\infty $};
\draw (325,215) node [anchor=north west][inner sep=0.75pt]    {$0$};
\draw (50,95) node [anchor=north west][inner sep=0.75pt]    {$V_{-}^{\lor }$};
\draw (285,95) node [anchor=north west][inner sep=0.75pt]    {$V_{+}^{\lor }$};
\draw (350,15) node [anchor=north west][inner sep=0.75pt]    {$V^{\lor }$};

\end{tikzpicture}
	\end{figure}
\end{setup}

\begin{remark}\label{remark:homoteties}
	The restriction of the $\C^*$-action on $V^\vee$ defined in Set-up \ref{setup:atiyah} to $V_-^\vee$ (resp. $V_+^\vee$) coincides with $\C^*_h$-action on $V^\vee_-$ (resp. $V^\vee_+$) by homoteties.
\end{remark}

 The fixed point locus of this action on $V^\vee$ coincides with the origin. 
 We can consider the induced $\C^*$-action of the coordinate ring of $V^\vee$, namely $\C[V^\vee]=\C\left[y_0,\ldots,y_m,x_0,\ldots,x_l\right]$.
 The GIT-quotient $V^\vee\to V^\vee\git \C^*:=\Spec \C[V^\vee]^{\C^*}$ is the affine cone over the Segre embedding of $\P\left(V_-\right)\times \P\left(V_+\right) \simeq \P^m \times \P^l$, therefore singular at the origin of $V_-^\vee \otimes V_+^\vee$.

\begin{remark}\label{remark:cobordism}
	Under the notation of Definition \ref{def: cobordism}, we have that
	\begin{align*}
	&B_- \simeq V^\vee\setminus \left\{y_0=\ldots=y_m=0\right\},\\
	&B_+\simeq V^\vee\setminus \left\{x_0=\ldots=x_l=0\right\}.
	\end{align*}
	In particular $B_{\pm}$ are non-empty open subsets of stable points under the $\C^*$-action, therefore we have two geometric quotients $B_{\pm}\to B_{\pm}/\C^*$.
\end{remark}

\begin{lemma}\label{lemma:atiyahexceptionallocus}
	The geometric quotients $B_-/\C^*$ and $B_+/\C^*$ are birational, and the exceptional locus of the birational map $\psi: B_-/\C^*\dashrightarrow B_+/\C^*$ is $\P\left(V_-\right)$. In particular, $\psi$ is a small modification.
\end{lemma}

\begin{proof}
	Notice that $B_-/\C^*$ and $B_+/\C^*$ are birational since they contain the open subset $(B_-\cap B_+)/\C^*\neq \emptyset$ 
	Let us study the exceptional locus of $\psi$. 
	It suffices to show that $V_-^\vee\setminus 0\subset B_-$ (then $\P\left(V_-\right)\subset B_-/\C^*$) and that $V_-^\vee\setminus 0\not\subset B_+$. 
	Let $p\in \P\left(V_-\right)$, and consider the associated line $\hat{p}$ in $V_-^\vee$ passing through the origin. Then $\hat{p}=\left(hv_-,0\right)$, for $h\in \C^*_h$ and $v_-\in V_-^\vee$. 
	For any point $q\in \hat{p} \setminus 0$ we have that that $\lim_{t\to \infty} t\cdot q$ does not exist, hence every point $q$ of $\hat{p}\setminus 0$ belongs to $B_-$. 
	By Remark \ref{remark:homoteties}, the restriction of the $\C^*$ on $V_-^\vee \setminus 0$ coincides with the action by homoteties, so we have that $p\in B_-/\C^*$. 
	On the other hand it is easy to see that $\lim_{t\to 0} t\cdot q$ exists, hence $q\notin B_+$.
	 Finally $\psi$ is a small modification since $\codim_{B_-/\C^*} \text{Exc}(\psi)\geq 2$. 
\end{proof}

\begin{remark}
	Let us recall that the example of Atiyah flip can be also easily described using toric geometry. In particular the birational map is induced by two different subdivisions of a cone. We refer to \cite[\S 3]{BR} for a detailed discussion from the toric point of view.
\end{remark}

\begin{corollary}
	The exceptional locus of $\psi^{-1}: B_+/\C^*\dashrightarrow B_-/\C^*$ is $\P(V_+)$.
\end{corollary}

\begin{remark}\label{remark:atiyahblowup}
	Let $\beta: W\to V^\vee\git \C^*$ be the blow-up of $V^\vee \git \C^*$ along the origin. Then the exceptional divisor is $\P(V_-)\times \P(V_+)$. 
\end{remark}

\begin{lemma}\label{lemma:atiyahfactorization}
	The blow-up $\beta: W\to V^\vee\git \C^*$ can be factorized through $b_{\pm}: W\to B_{\pm}/\C^*$ and the small contractions $s_{\pm}: B_{\pm}/\C^*\to V^\vee\git \C^*$ with exceptional loci $\P(V_{\pm})$. 
\end{lemma}

\begin{proof}
	There exist two natural birational morphisms $s_{\pm}: B_{\pm}/\C^*\to V^\vee\git \C^*$, isomorphisms over the set of points which are semistable but not stable. From this it follows that $\Exc(s_{\pm})=\P(V_{\pm})$. Since they have codimension greater than two, we conclude $s_{\pm}$ are small contractions. 
	Moreover, since $W$ is birational to $V^\vee\git \C^*$, which is birational to $B_{\pm}/\C^*$, we conclude there exist birational maps $b_{\pm}: W\dashrightarrow  B_{\pm}/\C^*$, defined over $W\setminus (\P(V_-)\times \P(V_+))$. Since the exceptional locus has codimension one, $b_{\pm}$ are morphisms.	
	Let us prove that $b_{\pm}\circ s_{\pm}=\beta$. For sake of simplicity, let us consider $b_-\circ s_-$. It is immediate to notice that $b_-^{-1}(s_-^{-1}(0))=b_-^{-1}(\P(V_-))=\P(V_-)\times \P(V_+) =\beta^{-1}(0)$. 
\end{proof}

We summarize the above construction by means of the diagram

\[
\xymatrix{ & W \ar[ld]_{b_-} \ar[rd]^{b_+} \ar[dd]^<<<<<<\beta & \\ B_-/\C^* \ar@{-->}[rr]^<<<<<<\psi \ar[rd]_{s_-} && B_+/\C^* \ar[ld]^{s_+} \\ & V^\vee \git \C^* & }
\]

and its restriction to the exceptional loci:
\[
\xymatrix{ & \P\left(V_-\right) \times \P\left(V_+\right) \ar[ld]_{b_-} \ar[rd]^{b_+}  & \\ \P\left(V_-\right) \ar[rd]_{s_-} && \P\left(V_+\right) \ar[ld]^{s_+} \\ & 0 &}
\]

\begin{theorem}
	The birational map $\psi:B_-/\C^*\dashrightarrow B_+/\C^*$ is a rooftop flip modeled by $\P(V_-)\times \P(V_+)$.
\end{theorem}

\begin{proof}
	We verify that each condition of Definition \ref{def:Dflips} is satisfied.
	For $(1)$, Using Lemma \ref{lemma:atiyahexceptionallocus}, the birational map $\psi: B_-/\C^*\dashrightarrow B_+/\C^*$ is a small modification, the exceptional loci of $s_{\pm}: B_{\pm}/\C^*\to \hat{X}\git \C^*$ are $\P(V_{\pm})$ and the restriction $s_{\pm}|_{\P(V_{\pm})}: \P(V_{\pm}) \to 0$ are obviously $\P(V_{\pm})$-bundles. 
	Let us prove verify $(2)$: if we consider the resolution $b_{\pm}: W\to B_{\pm}/\C^*$ we have that $\P(V_-)\times \P(V_+)=b^{-1}_{\pm}(\P(V_{\pm}))$ is a divisor, and $\P(V_-)\times \P(V_+)\to \P(V_{\pm})$ defines two projective bundle structures.
	Finally, we consider $(3)$: In this case $Z_0$ is the origin, and we know that $s^{-1}_{\pm}(0)\simeq \P(V_{\pm})$. Moreover $(b_{\pm}^{-1}\circ s_{\pm}^{-1})(0)\simeq \P(V_-)\times \P(V_+)$, hence we conclude.
\end{proof}

\begin{corollary}
	The geometric quotients $B_{\pm} /\C^*$ are smooth, and the rooftop flip $\psi: B_-/\C^*\dashrightarrow B_+/\C^*$ is in particular a small $\Q$-factorial modification.
\end{corollary}

\begin{proof}
	Since the non-empty open subset $B_{\pm}$ are smooth, and $\C^*$ acts freely on them, using \cite[Corollary p.199]{MFK} we obtain that $B_{\pm}$ are $\C^*$-principal bundles over $B_{\pm}/\C^*$, hence they are smooth. 
\end{proof}


\section{Main result}\label{sec:mainresult}
	
This section is devoted to prove Theorem \ref{theorem:naive}; we will adapt the example of Atiyah flip explained in Section \ref{sec:atiyah} in greater generality.

\begin{setup}\label{setup:maintheorem}
	Let $X$ be a smooth drum constructed upon a triple $(Y,\cL_-,\cL_+)$ (cfr. \S \ref{ssec:drums}).
	Let $\hat{X}$ be the affine cone over $X$, contained in the affine space $V^\vee:=V_-^\vee\oplus V_+^\vee$ (cfr. Remark \ref{remark:drumembedding}), where 
	\[
	V_-:=\HH^0(Y_-,L_-), \qquad V_+:=\HH^0(Y_+,L_+).
	\]
	The $\C^*$-action on $V^\vee$ is given by $t\cdot v=\left (t v_-,t^{-1}v_+\right)$, where $v=\left(v_-,v_+\right)\in V^\vee$.
\end{setup}
		
Notice that the $\C^*$-action on $V^\vee$ is induced by a $\C^*$-action on $V$ as described in Set-up \ref{setup:atiyah}. 

Consider the restriction of the $\C^*$-action on $V^\vee$ to $\hat{X}$; there exists a GIT quotient $\hat{X}\to \hat{X}\git\C^*$, singular at the origin. Moreover we can also consider the intersection $\hat{X}\cap B_{\pm}$, which are non-empty open subsets of stable points giving  geometric quotients $\pi_{\pm}:\hat{X}\cap B_{\pm}\to \hat{X}\cap B_{\pm}/\C^*$. 
	
\begin{proposition}\label{proposition:mainexceptionallocus}
	There exist a small modification \[\psi: \hat{X}\cap B_-/\C^*\dashrightarrow \hat{X}\cap B_+/\C^*,\] whose exceptional locus is $Y_-$. 
\end{proposition}

\begin{proof}
	The existence of such a birational map is immediate after noticing that $\hat{X}\cap B_-\cap B_+$ is open and non-empty.
	Let us prove that the exceptional locus of $\psi$ is $Y_-$. It suffices to show that 
	$$(\hat{X}\cap B_-/\C^*) \cap \P(V_-) =Y_-.$$
	Notice that the $\supset$ inclusion is trivial, so let us focus on the $\subset$ inclusion. Let $p\in (\hat{X}\cap B_-/\C^*) \cap \P(V_-)$, and let $\hat{p}=(hv_-,0)$ the corresponding line in $V_-^\vee$ through the origin, with $h\in \C^*_h$. The preimage $\pi^{-1}_-(p)$ is a closed orbit $\C^*\cdot q$ such that $\lim_{t\to \infty} t\cdot q$ does not exist, hence $\C^*\cdot q=\{(tv_-,v_+)\mid t\in \C^*\}$. Since $p$ belongs to the intersection, $\C^*\cdot q=\{(tv_-,0)\mid t\in \C^*\}$. Therefore since the restriction to $V^\vee_-$ of the $\C^*$action on $V^\vee$ coincides with the $\C^*_h$-action on $V^\vee_-$ by homoteties, we have that $\hat{p}=\C^*\cdot q$. We conclude since $\hat{X}\cap V_-^\vee=\hat{Y_-}$, therefore $p\in Y_-$. Since $\codim_{\hat{X}\cap B_{\pm}/\C^*} (Y) \geq 2$, we conclude.
\end{proof}

\begin{corollary}	
The exceptional locus of the birational map $\psi^{-1}$ is $Y_+$.
\end{corollary}
Consider the blow-up $\beta: W\to V^\vee\git \C^*$ along the vertex of the affine cone $V^\vee \git \C^*$, as in Proposition \ref{remark:atiyahblowup}, with exceptional divisor $\P(V_-)\times \P(V_+)$. 

\begin{definition}\label{definition:mainresolution}
	Let $R:=\overline{\beta^{-1}((\hat{X}\git \C^*)\setminus 0)}$ be the strict transform of $\hat{X}\git \C^*$ under $\beta: W\to V^\vee\git \C^*$. 
\end{definition}

We abuse notation by denoting with $b_{\pm}: R\to \hat{X}\cap B_{\pm}/\C^*$ the restriction of the blow-up $b_{\pm}: W\to B_{\pm}/\C^*$.	Notice that $R\simeq \overline{b_{\pm}^{-1}((\hat{X}\cap B_{\pm}/\C^*) \setminus \hat{Y_{\pm}})}$, where again we abuse notation by denoting with $s_{\pm}: \hat{X}\cap B_{\pm} \to \hat{X}\git \C^*$ the restriction of $s_{\pm}: B_{\pm}/\C^*\to V^\vee\git \C^*$.  We obtain a diagram:
\[
\xymatrix{ & R \ar[ld]_{b_-} \ar[rd]^{b_+} \ar[dd]^<<<<<<\beta & \\ \hat X \cap B_-/\C^* \ar@{-->}[rr]^<<<<<<\psi \ar[rd]_{s_-} && \hat X \cap B_+/\C^* \ar[ld]^{s_+} \\ & \hat X \git \C^* & }
\]

\begin{proposition}\label{proposition:mainexceptionaldivisor}
	It holds that $b_{\pm}^{-1}(Y_{\pm})\simeq Y$. 
\end{proposition}

\begin{proof}
	We proceed by steps. First, let us denote by $X/\C^*$ the geometric quotient of $\left(X,L\right)$ under the $\C^*$-action, defined over the set of stable points $X\setminus (Y_-\cup Y_+)$ (see \cite[Proposition 2.9]{WORS3})
	\begin{description}
	\item [Step 1] We want to prove that $Y\simeq X/\C^*$.
	 Thanks to \cite[Remark 4.2]{WORS1}, the contraction $f:\P(\cE)\to X$ is $\C^*$-equivariant, in particular the geometric quotients of $(\P(\cE),\cO_{\P(\cE)}(1))$ and $(X,L)$ with respect to the $\C^*$-action are isomorphic. Since the former is a $\P^1$-bundle on $Y$, and therefore its geometric quotient is isomorphic to $Y$, we conclude.
	 \item [Step 2] We show that the GIT quotient $\hat{X}\git \C^*$ is the affine cone over over $Y$.
	 Let us recall that by $\C^*_h$ we denote the natural $\C^*$-action on the affine space $V^\vee$ given by the homoteties. We claim that 
		\[
		\left(\hat{X}\git \C^* \setminus 0\right)/\C^*_h\simeq Y.
		\]
		To this end, let us note that the two $\C^*$-actions commute over the open subset of the points stable under both the $\C^*$ and the $\C^*_h$ actions.
		Therefore we have that
		\begin{equation}\label{eq:quotients}
		\left(\hat{X}\git \C^* \setminus 0\right)/\C^*_h \simeq \left(\hat X \setminus\left(\hat Y_- \cup \hat Y_+\right)\right)/\left(\C^*_h \times \C^*\right).
		\end{equation}
		Notice that 
		\[
		\frac{\hat X \setminus \left(\hat Y_- \cup \hat Y_+\right)}{\C^*_h}=\frac{\left( \hat X \setminus 0\right) \setminus \left( \left(\hat Y_- \setminus 0\right) \cup \left(\hat Y_+ \setminus 0\right) \right)}{\C^*_h}\simeq X \setminus \left(Y_- \cup Y_+\right)
		\]
		and that
		\[
		\left(X \setminus \left(Y_- \cup Y_+\right)\right)/\C^*=X/\C^* \simeq Y.
		\]
		Then the right-hand side of \eqref{eq:quotients} is isomorphic to $Y$ and we conclude.
	\item [Step 3] We want to prove that $\beta^{-1}(0)=Y$.
	It follows immediately after recalling that we are considering the restriction of the blow-up map to $\hat{X}\git \C^*$, whose base of the cone is precisely $Y$.
	\item [Step 4] We show that  $s_{\pm} \circ b_{\pm}=\beta$ and that $s_{\pm}^{-1}(0)\simeq Y$.
	 The first claim follows directly from Lemma \ref{lemma:atiyahfactorization}. Since $s_{\pm}: B_{\pm}/\C^*\to \hat{X}\git \C^*$ are small contractions whose exceptional locus is $\P(V_{\pm})\cap \hat{X}=Y_{\pm}$ by Proposition \ref{proposition:mainexceptionallocus}, we conclude. \qedhere
	\end{description}
\end{proof}

We can now rephrase Theorem \ref{theorem:naive} as follows:
	
\begin{theorem}\label{theorem:maintheorem}
	With the notation of Set-up \ref{setup:maintheorem}, for any smooth drum $X$ constructed upon $(Y,\cL_-,\cL_+)$ there exist a rooftop flip $\psi: \hat{X}\cap B_-/\C^*\dashrightarrow \hat{X}\cap B_+/\C^*$ modeled by $Y$. 
\end{theorem}
	
\begin{proof}
	We verify each condition of Definition \ref{def:Dflips} is satisfied.
	\begin{enumerate}
	\item Using Proposition \ref{proposition:mainexceptionallocus}, the birational map $\psi: \hat{X}\cap B_-/\C^*\dashrightarrow \hat{X}\cap B_+/\C^*$ is a small modification with exceptional locus $Y_-$. The exceptional loci of $s_{\pm}: \hat{X}\cap B_{\pm}/\C^*\to \hat{X}\git \C^*$ are $Y_{\pm}$ and the restriction $s_{\pm}|_{Y_{\pm}}: Y_{\pm} \to 0$ are precisely $Y_{\pm}$-bundles. 
	\item If we consider the resolution $b_{\pm}: R\to \hat{X}\cap B_{\pm}/\C^*$ we have that $Y=b^{-1}_{\pm}(Y_{\pm})$ is a divisor in $R$, and $Y\to Y_{\pm}$ defines two projective bundle structures by definition of smooth drum.
	\item  In this case $Z_0=0$, and we know that $s^{-1}_{\pm}(0)\simeq Y_{\pm}$. Moreover $(b_{\pm}^{-1}\circ s_{\pm}^{-1})(0)\simeq Y$ by Proposition \ref{proposition:mainexceptionaldivisor}, hence we conclude. \qedhere
	\end{enumerate}
\end{proof}

\begin{corollary}
	The geometric quotients $\hat{X}\cap B_{\pm} /\C^*$ are smooth and in particular the rooftop flip $\psi: \hat{X}\cap B_-/\C^*\dashrightarrow \hat{X}\cap B_+/\C^*$ is a small $\Q$-factorial modification.
\end{corollary}

\begin{proof}
	Since the affine variety $\hat{X}$ has only a singularity at the origin, $\hat{X}\cap B_{\pm}$ is smooth. Moreover, the $\C^*$-action is free on $\hat{X}\cap B_{\pm}$, therefore using \cite[Corollary p.199]{MFK} $\hat{X}\cap B_{\pm}$ is a $\C^*$-principal bundle over $\hat{X}\cap B_{\pm}/\C^*$, hence they are also smooth. By definition $\psi$ is in particular a small $\Q$-factorial modification.
\end{proof}

We conclude by using Theorem \ref{theorem:maintheorem} to show that the smooth drum $Q^{2n}\subset \P^{2n+1}$ induces a rooftop flip modeled by $\P\left(T_{\P^n}\right)$. Notice that for $n=2$ this is the famous \emph{Mukai flop} (see \cite{HZ04}, \cite{WW}):

\begin{corollary}\label{corollary:rooftopflipquadric}
	Let $Q^{2n}\subset \P^{2n+1}$ be the smooth quadric hypersurface, viewed as the smooth drum constructed upon $\left(\P\left(T_{\P^n}\right), p^{*}_-\cO_{\P^n}(1), p^{*}_+\cO_{(\P^n)^\vee}(1)\right)$ (cf. Example \ref{example:quadric}). Then the small modification
	$$\psi: \hat{Q}^{2n}\cap B_-/\C^* \dashrightarrow \hat{Q}^{2n}\cap B_+/\C^*$$
	is a rooftop flip modeled by $\P\left(T_{\P^n}\right)$.
\end{corollary}

\bibliographystyle{plain}
\bibliography{bibliomin}
\end{document}